\documentclass[10pt,a4paper]{article}
\usepackage[utf8]{inputenc}
\usepackage{amsmath}
\usepackage{amsfonts}
\usepackage{amssymb}
\usepackage{amsthm}
\usepackage[english]{babel}
\usepackage[utf8]{inputenc}
\usepackage[parfill]{parskip}
\usepackage{hyperref}
\usepackage[alphabetic]{amsrefs}
\usepackage{xcolor}\hypersetup{linkbordercolor=orange}

\begin{document}
\title{\textbf{Compactness of certain class of singular minimal hypersurfaces}}

\author{Akashdeep Dey \thanks{Email: adey@math.princeton.edu, dey.akash01@gmail.com}}
\date{}
\maketitle

\newtheorem{theorem}{Theorem}[section]
\newtheorem{lemma}[theorem]{Lemma}
\newtheorem{claim}[theorem]{Claim}
\newtheorem{proposition}[theorem]{Proposition}

\theoremstyle{definition}
\newtheorem{definition}[theorem]{Definition}
\newtheorem{example}[theorem]{Example}
\newtheorem{xca}[theorem]{Exercise}

\theoremstyle{remark}
\newtheorem{remark}[theorem]{Remark}

\numberwithin{equation}{section}

\newcommand{\mf}{manifold\;}
\newcommand{\mfs}{manifolds\;}
\newcommand{\mt}{metric\;}
\newcommand{\st}{such that\;}
\newcommand{\Th}{Theorem\;}
\newcommand{\te}{there exists\;}
\newcommand{\tf}{Therefore, \;}
\newcommand{\wrt}{with respect to\;}
\newcommand{\bbr}{\mathbb{R}}
\newcommand{\bbn}{\mathbb{N}}
\newcommand{\mres}{\mathbin{\vrule height 1.6ex depth 0pt width
		0.13ex\vrule height 0.13ex depth 0pt width 1.3ex}}

\begin{abstract}
	Given a closed Riemannian \mf $(N^{n+1},g)$, $n+1 \geq 3$ we prove the compactness of the space of singular, minimal hypersurfaces in $N$ whose volumes are uniformly bounded from above and the $p$-th Jacobi eigenvalue $\lambda_p$'s are uniformly bounded from below. This generalizes the results of Sharp \cite{Sharp} and  Ambrozio-Carlotto-Sharp \cite{ACS} in higher dimensions. 
\end{abstract}

\section{Introduction}

A hypersurface of a Riemannian \mf $(N^{n+1},g)$ is called \textit{minimal} if it is a critical point of the $n$-dimensional area functional. By the combined works of Almgren \cite{Alm}, Pitts \cite{Pitts} and Schoen-Simon \cite{SS} one gets the following Theorem.
\\
\begin{theorem} [\cite{Alm}, \cite{Pitts}, \cite{SS}]
	Let $(N^{n+1},g)$ be an arbitrary closed Riemannian manifold with $n+1 \geq 3$. Then $N$ contains a singular, minimal hypersurface which is smooth and embedded outside a singular set of Hausdorff dimension atmost $n-7$. In particular, when $3 \leq n+1 \leq 7$ \te a smooth, closed, embedded, minimal hypersurface in $N$. 
	\label{Almgren-Pitts min-max theorem}
\end{theorem}
\medskip
Recently, Almgren-Pitts min-max theory has been further developed to show that minimal hypersurfaces exist in abundance when the ambient dimension $3 \leq n+1 \leq 7$. By the results of Marques-Neves \cite{MN_ricci_positive} and Song \cite{song} every closed Riemannian \mf $N$ of dimension $3\leq n+1 \leq 7$ contains infinitely many minimal hypersurfaces. Moreover, Irie, Marques and Neves have shown that \cite{IMN} for a generic metric the union of all closed, minimal hypersurfaces is dense in $N$; this theorem was later improved by Marques, Neves and Song in \cite{MNS} where they proved that for a generic metric \te an equidistributed sequence of closed, minimal hypersufaces in $N$. The Weyl law for the
volume spectrum proved by Liokumovich, Marques and Neves \cite{LMN} played a major role in the arguments of \cite{IMN} and \cite{MNS}. There is yet another proof of the existence of infinitely many closed, minimal hypersurfaces for a generic \mt on $N$ which follows form the papers by Marques-Neves \cite{MN_ricci_positive} and Zhou \cite{Zhou2}. The reason of the upper bound of the dimension $n+1\leq 7$ in \cite{song}, \cite{IMN}, \cite{MNS}, \cite{Zhou2} is that the space of singular, minimal hypersurfaces is not well understood unlike the smooth case (\cite{white}, \cite{white2}).
\\\\
Using the Allen-Cahn equation Chodosh and Mantoulidis \cite{CM} have proved the existence of infinitely many minimal surfaces for generic metrics in dimension $3$; Gaspar and Guaraco \cite{GG} have given alternative proofs of the above mentioned density and equidistribution theorems.
\\\\
In higher dimensions, Li \cite{Li} has proved that a closed \mf $M^{n+1}$, $n+1 \geq8$ equipped with a generic Riemannian \mt contains infinitely many singular, minimal hypersurfaces with optimal regularity (i.e. the singular set has Hausdorff dimension atmost $n-7$).
\\\\
One of the key ingredients of the papers \cite{song}, \cite{IMN}, \cite{MNS}, \cite{Zhou2} is Sharp's compactness theorem \cite{Sharp} which asserts certain compactness properties of the set of smooth, closed, minimal hypersurfaces in a Riemannian \mf $(M^{n+1},g)$, $3\leq n+1 \leq 7$ with bounded volume and index. This result was generalized by Ambrozio, Carlotto and Sharp \cite{ACS} where instead of bounded volume and index, an upper bound of the volume and a lower bound of the $p$-th Jacobi eigenvalue $\lambda_p$ (for some $p\in \mathbb{N}$) was assumed. (We note that for a smooth, closed minimal hypersurface $\Sigma$, Ind$(\Sigma) \leq I$ is equivalent to $\lambda_{I+1}(\Sigma)\geq 0$.)
\\\\
In the present article we will suitably generalize the results of \cite{Sharp} and \cite{ACS} in higher dimensions; for that, we need to consider the minimal hypersurfaces which may have singularities. We will state the notion of the index and the $p$-th Jacobi eigenvalue for a stationary $n$-varifold  and prove the following Theorem.
\\
\begin{theorem}
	Let $\{M_k\}_{k=1}^{\infty}$ be a sequence of closed, connected, singular, minimal hypersurfaces in a closed Riemannian \mf $(N^{n+1},g)$, $n+1 \geq 3$. Let $V_k= |M_k|$, the varifold associated to $M_k$. Suppose, there exist $\Lambda>0$, $\alpha\geq 0$, $p\in \bbn$ such that for all $k$
	\begin{itemize}
	    \item $\mathcal{H}^{n-2}(\text{sing}(M_k))=0$
		\item $\mathcal{H}^n(M_k)=\|V_k\|(N)\leq \Lambda$
		\item $\lambda_p(V_k)\geq -\alpha$
	\end{itemize}
	Then there is a stationary, integral varifold  $V$ \st  possibly after passing to a subsequence, $V_k \longrightarrow V$ in the \textup{\textbf{F}} \mt. Moreover, denoting $M=\text{spt}(V)$ we have
	\begin{itemize}
	    \item $\|V\|(N)\leq \Lambda$
		\item $\lambda_p(V)\geq -\alpha$
		\item $\mathcal{H}^{s}(\text{sing}(M))=0 \quad \forall s >n-7$
		\item The convergence is smooth and graphical over the compact subsets of $\text{reg}(M) \setminus \mathcal{Y}$ where $\mathcal{Y}$ is a finite subset of $reg(M)$ with $|\mathcal{Y}|\leq p-1$.
	\end{itemize} 
	\label{main theorem to prove}
\end{theorem}
\medskip

From the definitions of the index and the Jacobi eigenvalue, it will be clear that $\textrm{Ind}(V) \leq I$ if and only if $\lambda_{I+1}(V) \geq 0$. \tf Theorem \ref{main theorem to prove} has the following Theorem as a corollary which generalizes Sharp's compactness theorem \cite{Sharp} in higher dimensions.
\\
\begin{theorem}
	Let $\{M_k\}_{k=1}^{\infty}$ be a sequence of closed, connected, singular, minimal hypersurfaces in a closed Riemannian \mf $(N^{n+1},g)$, $n+1 \geq 3$. Suppose for all $k$, $\mathcal{H}^{n-2}(\text{sing}(M_k))=0$, $\mathcal{H}^n(M_k)\leq \Lambda$ and \textup{Ind}$(|M_k|) \leq I$. Then there is a stationary, integral varifold $V$ \st  possibly after passing to a subsequence, $|M_k| \longrightarrow V$ in the \textup{\textbf{F}} \mt, $\|V\|(N)\leq \Lambda$ and \textup{Ind}$(V)\leq I$. Further, if $M= \text{spt}(V)$ then $\mathcal{H}^{s}(\text{sing}(M))=0 \quad \forall s >n-7$ and the convergence is smooth and graphical over the compact subsets of $\text{reg}(M) \setminus \mathcal{Y}$ where $\mathcal{Y}$ is a finite subset of $reg(M)$ with $|\mathcal{Y}|\leq I$. \label{sharp}
\end{theorem}
\medskip

The space of minimal hypersurfaces with bounded volume and index is particularly interesting; due to the work of Marques and Neves \cite{MN_index}, the minimal hypersurfaces constructed by the min-max procedure have bounded volume and index. More precisely, they have proved the following Theorem.
\\
\begin{theorem}[\cite{MN_index}]
	Suppose $(N^{n+1},g)$ is a closed Riemannian manifold, $n+1 \geq 3$. Let $X$ be an $m$ dimensional simplicial complex and $\Pi$ be a $\mathcal{F}$-homotopy class of continuous maps from $X$ to $\mathcal{Z}_n(N;\textup{\textbf{F}};\mathbb{Z}_2)$. We define
	$$\textup{\textbf{L}}(\Pi) = \inf_{\Phi \in \Pi} \sup_{x \in X} \left\{\textup{\textbf{M}}(\Phi(x))\right\}$$
	Then there is a stationary, integral varifold $V$ with $spt(V)=\Sigma$ \st
	\begin{itemize}
		\item $\|V\|(N)=\textup{\textbf{L}}(\Pi)$
		\item \textup{Ind}$(V) \leq m$ 
		\item $\mathcal{H}^{s}(\text{sing}(\Sigma))=0 \quad \forall s >n-7$.
	\end{itemize}
\end{theorem}
\bigskip

The index upper bound of the minimal hypersurfaces in the Allen-Cahn settings has been proved by Gaspar \cite{Gaspar} and Hiesmayr \cite{H}.
\\\\
If we take $M_k$ to be $M$ for all $k$ in Theorem \ref{main theorem to prove}, we get the following regularity result.
\\
\begin{proposition}
	
	Let $M^n$ be a singular, minimal hypersurface in $(N^{n+1},g)$, $n+1 \geq 3$. Suppose, $\mathcal{H}^{n-2}(\text{sing}(M))=0$ and $\lambda_p(|M|) > -\infty$  for some $p$. Then $\mathcal{H}^{s}(\text{sing}(M))=0 \;\forall s >n-7$.
	
\end{proposition}
\medskip
The proof of Theorem \ref{main theorem to prove} is very similar to that of \cite{Sharp} and \cite{ACS}. However, for the sake of completeness we will give a self-contained proof of it.
\\\\
\textbf{Acknowledgements.} I am very grateful to my advisor Prof. Fernando Codá Marques for many helpful discussions and for his support and guidance. I also thank Yangyang Li and Antoine Song for answering some of my questions. The author is partially supported by NSF grant DMS-1811840.

\medskip

\section{Notations and Preliminaries}

\subsection{Notations} Here we summarize the notations which will be frequently used later.
\[
\begin{array}{ll}
	\mathcal{V}_n(U) & \text{the set of } n-\text{varifolds in } U \\ 
	I\mathcal{V}_n(U) & \text{the set of integral } n-\text{varifolds in } U \\
	\mathcal{H}^s & \text{the Hausdorff measure of dimension } s \\
	\|V\| & \text{the Radon measure associated to the varifold } V \\ 
	|\Sigma| & \text{the varifold associated to a singular hypersurface } \Sigma \\
	\delta^2 V & 2\text{-nd variaration of the stationary varifold } V \\
	\text{Ind}(.) & \text{index (of a stationary hypersurface or varifold)}\\
	\lambda_k(.) & k-\text{th Jacobi eigenvalue (of a stationary hypersurface or varifold)}\\
	sing(\Sigma) & \text{the singular part of } \Sigma\\
	reg(\Sigma) & \text{the regular part of } \Sigma\\
	B(p,r) & \text{open ball of radius } r \text{ centered at }p
\end{array}
\] 

\subsection{Preliminaries from geometric measure theory}
Here we will briefly discuss the notion of varifold and various related concepts; further details can be found in Simon's book \cite{Sim}.
\\\\
Given a Riemannian \mf $(U^{n+1},g)$ let $G_k(U)$ denote the Grassmanian bundle of $k$-dimensional hyperplanes over $U$. A $k$\textit{-varifold} in $U$ is a positive Radon measure on $G_k(U)$. The topology on the space of $k$-varifolds $\mathcal{V}_k(U)$ is given by the weak* topology i.e. a net $\{V_i\}_{i \in I} \subset \mathcal{V}_k(U)$ converges to $V$ iff
$$\int _{G_k(U)} f(x, \omega)dV_i(x, \omega) \longrightarrow \int _{G_k(U)} f(x, \omega)dV(x, \omega)$$
for all $f \in C_c(G_k(U))$. This topology is metrizable and the metric is denoted by \textbf{F}. If $V \in \mathcal{V}_k(U)$ and $\pi:G_k(U) \longrightarrow U$ denotes the canonical projection then $\|V\| = \pi_{*}V$ is a Radon measure on $U$; $\|V\|(A)=V(\pi^{-1}(A))$.
\\\\
If $\varphi : U \longrightarrow U'$ is a diffeomorphism and $ V \in \mathcal{V}_k(U)$, we define $\varphi_{*}V \in \mathcal{V}_k(U')$ by the following formula

$$(\varphi_{*}V)(g) = \int _{G_k(U)}g\left(\varphi(x),d_x \varphi(\omega)\right)J\varphi(x, \omega)dV(x, \omega)$$
where 
$$J\varphi(x, \omega) = \left(\det\left(\left(d_x \varphi \big |_{\omega}\right)^t \circ \left(d_x \varphi \big |_{\omega}\right)\right)\right)^{1/2}$$

is the Jacobian factor and $g \in C_c(G_k(U'))$. Given a compactly supported, smooth vector-field $X$ on $U$ let $\varphi_t$ denote the flow of $X$; \textit{the first variation and second variation of} $V$ are given by

$$\delta V(X) = \frac{d}{dt} \bigg |_0 \|(\varphi_t)_{*}V \|(U) \qquad ; \qquad
\delta^2 V(X,X) = \frac{d^2}{dt^2} \bigg |_0 \|(\varphi_t)_{*}V \|(U)$$

We say that \textit{$V$ is stationary} if $\delta V (X)= 0$ for all $X$ and a stationary varifold $V$ is called \textit{stable} if $\delta^2V(X,X) \geq 0$ for all $X$.
\\\\
Given a $k$-rectifiable set $S \subset U$ and a non-negative function $ \theta \in L^1_{loc}(S, \mathcal{H}^k \mres S)$ we define the $k$-varifold $\textbf{v}(S, \theta)$ by
$$\textbf{v}(S, \theta)(f)= \int_S f(x, T_xS) \theta(x) d \mathcal{H}^k(x)$$
where $T_xS$ denotes the tangent space of $S$ at $x$ which exists $\mathcal{H}^k \mres S$-a.e. $V$ is called an \textit{integral $k$-varifold} if $V = \textbf{v}(S, \theta)$ for some $S$ and $\theta$ with $\theta$ taking non-negative integer values $\mathcal{H}^k \mres S$-a.e.
\\\\
In the present article we will only deal with $n$-varifolds and from now on we will simply write `varifold' instead of `$n$-varifold'. Given $A \subset U$ we define the regular and singular part of $A$
\begin{align*}
reg(A)=\{x \in A :\; &
\exists \text{ open } P \subset U \text{ containing } x \text{ such that } \\
& P \cap A \text{ is a smooth, embedded hypersurface} \}
\end{align*} 
and 
$$ sing(A) = A \setminus reg(A).$$

Further, by a singular, minimal hypersurface $\Sigma$ in $U$ we will mean that $\Sigma \subset U$ is a closed, $n$-rectifiable set with $\mathcal{H}^{n-1}(sing(\Sigma))=0$ and $|\Sigma|=\textbf{v}(\Sigma , \textbf{1}_{\Sigma})$ (where $\textbf{1}_{\Sigma}$ is the constant function $1$) is stationary. By \cite{Ilmanen} (Equation $(4)$), $|\Sigma|$ is stationary in $U$ if and only if $reg(\Sigma)$ is a smooth, minimal hypersurface and $\mathcal{H}^n(\Sigma \cap B(x,r)) \leq C(U')r^n$ for all $B(x,r) \subseteq U' \subset \subset U$.

\section{Index and Jacobi eigenvalues of a stationary varifold}

We will now state the notion of the index and the Jacobi eigenvalues of a stationary varifold following the paper by Marques and Neves \cite{MN_index}. The definition is motivated by the following min-max characterization of $\lambda_k(\Sigma)$ when $\Sigma^n \subset (U^{n+1},g)$ is a smooth, minimal hypersurface.

$$ \lambda_k(\Sigma)=\inf_{\dim (V)=k} \sup_{X \in V \setminus \{0\}} \left(\frac{\delta^2\Sigma(X,X)}{\int_{\Sigma}|X|^2 d\mathcal{H}^n}\right) $$

The infimum is over all the $k$-dimensional linear subspaces $V \subset \Gamma_c(N\Sigma)=$ compactly supported smooth sections of $N\Sigma$. \tf $\lambda_k(\Sigma) <a$ iff there is a $k$-dimensional subspace $V \subset \Gamma_c(N\Sigma)$ \st for all $X \in V \setminus \{0\}$, 
$$ \delta^2\Sigma (X,X) < a \int_{\Sigma}|X|^2 d\mathcal{H}^n $$
\\\\
Given a Riemannian \mf $(U^{n+1},g)$ and $k \in \mathbb{N}$ ($0 \notin \bbn$), \textit{a k-parameter family of diffeomorphisms} is a smooth map $F: \overline{B}^k(0,1)(\subset \bbr^k) \longrightarrow \text{Diff}(U)$ \st 
\begin{itemize}
\item $F_{v}(=F(v))=(F_{-v})^{-1} \;\; \forall v \in \overline{B}^k(0,1)$ and $F_0 = \text{Id}$
\item \te open $U'\subset \subset U$ \st $F_v \big |_{U\setminus U'} = \text{Id}\;\; \forall v \in \overline{B}^k(0,1)$
\end{itemize} 
If $F$ is a $k$-parameter family of diffeomorphisms, we define the vector-fields $Y_i$, $i = 1,...,k$ by 
$$Y_i \big | _p = \frac{d}{dt}\bigg |_0F_{te_i}(p)$$ 
Suppose further, we have a stationary varifold $V$ in $U$. Then we define a smooth function $A^V$ and a quadratic form $K^V$ as follows.
$$A^V:\overline{B}^k(0,1) \longrightarrow [0, \infty), \qquad  A^V(v)=\|(F_v)_{\#}V\|(U)$$
$$K^V: \bbr^k \times \bbr^k \longrightarrow \bbr, \qquad K^V(u,u)= \big\|\sum_i u_i Y_i\big\|_{L^2(U, \|V\|)}^2$$

\medskip
\begin{remark}
	If $V_i \longrightarrow V$ in the $\textbf{F}$ \mt then $A^{V_{i}} \longrightarrow A^V$ (\cite{Pitts}, Section 2.3) in the smooth topology and also $K^{V_{i}} \longrightarrow K^V$ smoothly on compact subsets.
	\label{A^V_i tends to A^V}
\end{remark}
\bigskip 

\begin{definition}
Given a  stationary varifold $V$ in $U$, $k \in \bbn$ and $\alpha \geq 0$ we say that $\lambda_k(V) < -\alpha $ if \te a $k$-parameter family of diffeomorphisms $F$ \st
$$ \left. D^2A^{V}\right\vert_0(u,u)<-\alpha K^V(u,u)\;\;$$
for all $ u \in \mathbb{R}^k \setminus \{0\} $ or equivalently for all $u \in \mathbb{R}^k$ with $\|u\|=1$.
Further, by $\lambda_k(V)\geq -\alpha $ we will mean that $\lambda_k(V) < -\alpha $ does not hold. If $\lambda_k(V) < -\alpha $ then restricting $F$ to $\overline{B}^{k-1}(0,1) \subset \overline{B}^k(0,1)$ we get that $\lambda_{k-1}(V) < -\alpha $ as well. Therefore, it will be natural to define 
\begin{equation*}
 \textup{Ind}(V) = 
 \begin{cases}
  0 & \text{if} \;\; \{I \in \bbn : \lambda_I(V) <0 \} = \emptyset, 
 \\
 \sup\{I \in \bbn : \lambda_I(V) <0 \} & \text{otherwise}.
 \end{cases}  
\end{equation*}
 
Hence, $\text{Ind}(V)\leq I$ is equivalent to $\lambda_{I+1}(V)\geq 0$. Further, $\text{Ind}(V)=0$ iff $\lambda_1(V) \geq 0$ iff $V$ is stable.
\label{Def of jacobi eigenvalue}
\end{definition}
\bigskip
\begin{remark}
	By Remark \ref{A^V_i tends to A^V} and from the above definition it is clear that whenever $\lambda_k(V) < -\alpha$ and $\textbf{F}(V,V')$ is sufficiently small, we have $\lambda_k(V') < -\alpha$ as well.
	\label{semicontinuity of eigenvalue}
\end{remark}
\bigskip
\begin{proposition}
	Given $\Lambda>0$, $k \in \bbn$ and $\alpha \geq 0$, the following sets are compact with respect to the \textup{\textbf{F}} metric topology.
	$$\mathcal{M}_U(\Lambda, k, \alpha)=\{V \in \mathcal{V}_n(U): V \text{ is stationary, } \|V\| \leq \Lambda \text{ and } \lambda_k(V) \geq -\alpha \} \subset \mathcal{V}_n(U)$$
	$$\mathcal{M'}_U(\Lambda, k, \alpha)=\{V \in I\mathcal{V}_n(U):  V \text{ is stationary, } \|V\| \leq \Lambda \text{ and } \lambda_k(V) \geq -\alpha \} \subset I\mathcal{V}_n(U)$$
\end{proposition}

\begin{proof}
	By the standard compactness theorems, if $\{V_i\}_{i=1}^{\infty}$ is a sequence of stationary varifolds with $\|V_i\| \leq \Lambda$ then upto a subsequence $V_i$ converges to a stationary varifold $V$ in the \textbf{F} \mt with $\|V\| \leq \Lambda$. Further, by Allard's theorem \cite{Allard} if $V_i$ 's are integral varifolds, $V$ is also an integral varifold. Moreover, by Remark \ref{semicontinuity of eigenvalue} if $\lambda_k(V) < -\alpha$ then for $i$ large $\lambda_k(V_i) < -\alpha $ as well. Hence, $\lambda_k(V) \geq -\alpha$.
\end{proof}
\bigskip
\begin{theorem}
Let $\Sigma$ be a singular, minimal hypersurface in $U$, $V=|\Sigma|$ and $\mathcal{H}^{n-2}(sing(\Sigma))=0$. Then, $\lambda_k(V) < -\alpha \leq 0 \Longleftrightarrow \lambda_k(reg (\Sigma)) < -\alpha$. Hence, $ \textup{Ind}(V) = \textup{Ind}(reg(\Sigma))$; therefore, $V$ is stable iff $reg(\Sigma)$ is stable.
\label{equiv to smooth}
\end{theorem}

\begin{proof}
We note that 
\begin{equation}
\left. D^2A^{V}\right\vert_0(u,u)= \left. \frac{d^2}{dt^2}\right\vert_0 A^V(tu,tu)= \left. \frac{d^2}{dt^2}\right\vert_0 \|(F_{tu})_\#V\|(U)=\delta^2V(\sum_i u_i Y_i,\sum_i u_i Y_i)\label{Hess A and 2nd var}
\end{equation}

Let $\lambda_k(reg (\Sigma))< -\alpha $. Then there are $k$ linearly independent, compactly supported normal vector-fields on $reg(\Sigma)$ say $X_1,X_2,...,X_k$ \st $$\delta^2V (X, X)< -\alpha \|X\|_{L^2(\Sigma)}^2$$  for any non-zero vector-field $X$ in the span of $\{X_i\}_{i=1}^k$. We extend each $X_i$ to a compactly supported, globally defined vector-field on $U$ and continue to call it by $X_i$; we define 
$ F_v=\Phi^{\sum_i v_iX_i}$
where $\Phi^X$ denotes the time $1$ flow of the vector-field $X$. Let us check that this choice of $F$ indeed works. Clearly, $F_{-v}=F_v^{-1}$.
\\\\
By \ref{Hess A and 2nd var} 
\begin{equation*}
\left. D^2A^{V}\right\vert_0(u,u)= \delta^2V (\sum_i u_iX_i, \sum_i u_iX_i)< -\alpha \|\sum_i u_iX_i\|_{L^2(\Sigma)}^2
\end{equation*}

For the converse, we consider the following.
\begin{equation*}
 \delta^2V(X,X)=\int_\Sigma \left((\text{div}_{\Sigma}\;X)^2+\sum_{i=1}^n \big|(\nabla_{\tau_i}X)^{\perp}\big|^2 - \sum_{i,j=1}^n \langle \tau_i, \nabla_{\tau_j}X\rangle\langle\tau_j , \nabla_{\tau_i}X\rangle \right)d\mathcal{H}^n
\end{equation*}
where $\{\tau_i\}_{i=1}^n$ is an orthonormal basis of $T_x\Sigma$. Since $\mathcal{H}^{n-2}(sing(\Sigma))=0$,
given any $\delta >0$ and $0<\kappa<1$ we can choose balls $\{B(x_i,r_i)\}_{i=1}^K$ \st each $r_i<\kappa$,
$sing(\Sigma)\subset \cup_{i}B(x_i, r_i)$ and $\sum_i r_i^{n-2}< \delta$. \tf $\sum_i r_i^{n-1}< \delta$ 
and $\sum_i r_i^{n}< \delta$ as well. We choose smooth cut-off functions $0\leq\zeta_i\leq 1$ on $U$ \st 
\begin{equation*}
\zeta_i =
 \begin{cases}
  0 &\text{on } B(x_i,r_i)
  \\
  1 &\text{outside } B(x_i, 2r_i)
 \end{cases} 
\end{equation*}
and $|\nabla \zeta_i|\leq 2/r_i$ (this can be ensured by choosing $\kappa$ sufficiently small). Let $\zeta_{\delta}=\min_i \zeta_i$. From the second variation formula, we see that \footnote{Though $\zeta_{\delta}$ is only a Lipschitz continuous function, its use in the subsequent calculations can be justified by an approximation argument. }
\begin{equation}
\bigg |\delta^2V(X,X)-\delta^2V(\zeta_{\delta}X,\zeta_{\delta}X) \bigg| \leq \int_{\Sigma} \left((1-\zeta_{\delta}^2)+|\nabla \zeta_{\delta}|+ |\nabla \zeta_{\delta}|^2 \right)f(X, \nabla X)d\mathcal{H}^n \label{cut off error}
\end{equation}

where $f$ is an expression involving $X$ and $\nabla X$. By the monotonicity formula, the R.H.S. of this equation is bounded by $C\delta$ for some constant $C$ depending only on $(U,g), n, \|V\|(U)$ and $\|X\|_{C^1}$. 
\\\\
Therefor, for $u \in \mathbb{R}^k$ with $\|u\|=1$
\begin{align}
\delta^2V \left(\zeta_{\delta}\sum_i u_i Y_i^{\perp}, \zeta_{\delta}\sum_i u_i Y_i^{\perp}\right) 
& = \delta^2V\left(\zeta_{\delta}\sum_i u_iY_i, \zeta_{\delta}\sum_i u_iY_i\right)\\
& \leq \delta^2V\left(\sum_i u_iY_i, \sum_i u_iY_i \right)+C\delta 
\label{cor to cut off error}
\end{align}
Here $C$ depends only on $(U,g), n, \|V\|(U)$ and $\|Y_i\|_{C^1}$ and not on $u$.
\\\\
We assume $\lambda_k(V) < -\alpha$. By Definition \ref{Def of jacobi eigenvalue} and equation \ref{Hess A and 2nd var} for all $u \in \bbr^k$ with $\|u\|=1$, 
\begin{equation}
\delta^2V(\sum _i u_iY_i, \sum _i u_iY_i)<-\alpha \|\sum _i u_iY_i\|_{L^2(\Sigma)}^2\; \Longrightarrow \sum _i u_iY_i \neq 0
\label{eqn linear independence of Y_i}
\end{equation}
\tf 
\begin{equation}
\sup_{\|u\|=1} \frac{\delta^2V(\sum _i u_iY_i, \sum _i u_iY_i)}{\|\sum _i u_iY_i\|_{L^2}^2}\leq -\alpha-2\varepsilon
\end{equation}
for some $\varepsilon >0$. Hence, using \ref{cor to cut off error} for all $\|u\|=1$ and $\delta$ sufficiently small
\begin{align}
&\delta^2V\left(\zeta_{\delta}\sum_i u_iY_i^{\perp}, \zeta_{\delta}\sum_i u_iY_i^{\perp}\right)\nonumber\\
& \leq \delta^2V \left(\sum_i u_iY_i, \sum_i u_iY_i\right)+C\delta \nonumber\\
& \leq (-\alpha-2\varepsilon)\int_{\Sigma}\big|\sum _i u_iY_i\big|^2 d\mathcal{H}^n +C\delta \nonumber\\
& \leq (-\alpha-2\varepsilon)\int_{\Sigma}\big|\sum _i u_iY_i\zeta_{\delta}\big|^2 d\mathcal{H}^n +C\delta \nonumber\\
& < (-\alpha-\varepsilon)\int_{\Sigma}\big|\sum _i u_iY_i\zeta_{\delta}\big|^2 d\mathcal{H}^n \label{justify}\\
& \leq (-\alpha-\varepsilon)\bigg\|\sum_{i}u_iY_i^{\perp}\zeta_{\delta}\bigg\|_{L^2(\Sigma)}^2 \label{index penultimate}
\end{align}
To justify \ref{justify}, we observe:
\begin{align*}
& \int_{\Sigma}\big|\sum _i u_iY_i\zeta_{\delta}\big|^2 d\mathcal{H}^n
= \int_{\Sigma}\big|\sum _i u_iY_i\big|^2 d\mathcal{H}^n + \int_{\Sigma}(\zeta_{\delta}^2-1)\big|\sum _i u_iY_i\big|^2 d\mathcal{H}^n \\
& \geq \inf_{\|u\|=1}\int_{\Sigma}\big|\sum _i u_iY_i\big|^2 d\mathcal{H}^n - C \sum_{i}r_i^n \;\; (\text {using monotonicity formula}) \\
& \geq \theta -C\delta \;\; (\text{for some }  \theta>0 \text{ by } \ref{eqn linear independence of Y_i} \text{ and } C \text{ is independent of }u)
\end{align*}

Clearly \ref{index penultimate} implies that $\{\zeta_{\delta}Y_i^{\perp}\}_{i=1}^k$ are linearly independent normal vector fields on $reg(\Sigma)$ and $\lambda_k(reg(\Sigma))< -\alpha$
\end{proof}

\medskip

In view of the above Theorem \ref{equiv to smooth}, we will use the terms $\lambda_k(V)$ and $\lambda_k(reg(\Sigma))$ interchangeably.


\section {Modifications of the results of Schoen-Simon \cite{SS} }

Suppose the unit ball $B^{n+1}(0,1) \subset \bbr^{n+1}$ is equipped with a Riemannian \mt $g$; $\mu_1$ is a constant \st if $g=g_{ij}dx^idx^j$   
\begin{equation}
\sup_{B^{n+1}(0,1)}\left |\frac{\partial g_{ij}}{\partial x_k}\right|\leq \mu_1 
\qquad
\sup_{B^{n+1}(0,1)} \left|\frac{\partial^2 g_{ij}}{\partial x_k \partial x_l}\right|\leq \mu_1^2
\label{C^2 bound on metric}
\end{equation}
\bigskip
\begin{theorem}[Modification of Schoen-Simon \cite{SS}, Theorem 1 ; page 747] Suppose $\Sigma$ is a singular, minimal hypersurface in $(B^{n+1}(0,1),g)$ satisfying\\ $\mathcal{H}^{n-2}(sing(\Sigma))=0$, $\mathcal{H}^n(\Sigma)\leq \mu $ and $\lambda_1(|\Sigma|)\geq -\alpha$ for some $\alpha \geq 0$. Then there exist $\delta_0 \in (0,1)$, $r_0 \in (0, 1/4)$ and $c>0$  depending only on $n, \mu, \mu_1, \alpha$ \st the following holds. If $ x=(x',x_{n+1})\in \Sigma \cap B^{n+1}(0,1/4) $, $\rho \leq r_0$, $\Sigma'$ is the connected component of $\Sigma \cap C(x, \rho)$ \textup{(}$C(x,\rho)$ is the cylinder on $B^n(x, \rho)$\textup{)} containing $x$ and 
	$$\sup_{y= (y', y_{n+1})\in \Sigma'}|y_{n+1}-x_{n+1}|\leq \delta_0 \rho \qquad(*)$$
	
	then, $\Sigma'\cap C(x,\rho/2 )$ consists of disjoint union of graphs of functions $u_1 < u_2 < ... < u_k$ defined on $B^n(x, \rho/2)$ satisfying the following estimate.
	
	$$\sup_{B^n(x,\frac{\rho}{2})}(|Du_i|+\rho |DDu_i|)\leq c\delta_0$$
	
	for $i = 1,2,..., k$. In particular, $\Sigma$ is smooth near $x$ and the second fundamental form $|A_\Sigma(x)|\leq c/\rho$ (for possibly a different constant $c$). \label{modification of SS thm-1}
\end{theorem}
\medskip
\begin{remark}
	The difference between the above Theorem \ref{modification of SS thm-1} and Theorem 1 of Schoen-Simon \cite{SS} is that instead of assuming $|\Sigma|$ is stable we have assumed that $\lambda_1(|\Sigma|)\geq -\alpha$. Indeed under this weaker assumption, the stability inequality $(1.17)$ of Schoen-Simon \cite{SS} (page 746) continues to hold with the constant $c_5$ replaced by $c_5+\frac{\alpha}{\mu_1^2}$; therefore all the successive calculations in the paper \cite{SS} go through. 
\end{remark}
\bigskip
\begin{theorem} [Modification of Schoen-Simon \cite{SS}, Theorem 2; page 784]
	
	Let $\{\Sigma_q\}$ be a sequence of singular, minimal hypersurfaces in $(B^{n+1}(0,1),g)$ \st $ \mathcal{H}^{n-2}(sing(\Sigma_q))=0$, $\lambda_1(|\Sigma_q|)\geq -\alpha$ and $|\Sigma_q|$ converges to a varifold $W$; $0 \in \text{spt}(W)=\Sigma$. Then $\mathcal{H}^{s}\left(sing(\Sigma)\cap B^{n+1}(0, \frac{1}{2})\right)=0$ for all $ s > n-7$.
	\label{modification of SS thm-2}
	
\end{theorem}

\begin{proof} 
	As before, The difference between the above Theorem \ref{modification of SS thm-2} and Theorem 2 of Schoen-Simon \cite{SS} is that instead of assuming $|\Sigma|$ is stable we have assumed that $\lambda_1(|\Sigma|)\geq -\alpha$.
	\\\\
	The proof of Theorem 2 of Schoen-Simon \cite{SS} goes as follows. By the successive blow-up argument, one arrives at a varifold $W_l$ which is a stationary, integral, codimension $1$ cone in $\mathbb{R}^{n-l+1}$ \st $sing(W_l)= \{0\}$ and 
	
	\begin{equation}
	\mathbb{R}^l \times W_l = \lim_{m \rightarrow \infty} J_{\#}\circ \tau_{y_m \#}\circ \mu_{r_m \#}|\Sigma_{q_m}| \label{tgt cone}
	\end{equation} 
	for some sequence of points $y_m \in \bbr^{n+1}$ and positive real numbers $r_m$ and some subsequence $ \{\Sigma_{q_m}\} \subset \{\Sigma_q\}$; $|y_m|, r_m \rightarrow \infty$. Here $J$ is some orthogonal transformation of $\mathbb{R}^{n+1}$, $\tau_y$ denotes the translation of $\mathbb{R}^{n+1}$ which brings $y$ to the origin, $\mu_r$ is the multiplication (scaling) by $r$. It is shown that (equation (5.22) of Schoen-Simon) $l \geq s$ for every $ s $ \st $\mathcal{H}^{s}\left(sing(\Sigma)\cap B^{n+1}(0, \frac{1}{2})\right)>0$ (hence, one only needs to show that $l\leq n-7$) 
	and $l\leq n-3$. Upto this point, the only facts about $\Sigma_{q}$ which are used : $\Sigma_q$ satisfies the stability inequality $(1.17)$ and Theorem 1 of Schoen-Simon \cite{SS}. After this, 
	stability of $|\Sigma_q|$ is used to conclude that $\mathbb{R}^l\times W_l$ is stable. Hence, $W_l$ is a smooth, stable codimension 1 cone in $\mathbb{R}^{n-l+1}$ with a singularity at origin. \tf $n-l \geq 7\;$  i.e. $l\leq n-7$.
	\\\\
	In our context of Theorem \ref{modification of SS thm-2} the above mentioned proof can be modified as follows. As noted above, stability inequality $(1.17)$ and Theorem 1 of Schoen-Simon \cite{SS} continue to hold under the weaker assumption $\lambda_1(|\Sigma_q|)\geq -\alpha$. Moreover
	\begin{equation}
	\lambda_1(|\Sigma_q|)\geq -\alpha \; \Longrightarrow
	\lambda_1\left(J_{\#}\circ \tau_{y_m \#}\circ \mu_{r_m \#}|\Sigma_{q_m}|\right)\geq \frac{-\alpha}{r_m}
	\label{eval conv to 0}
	\end{equation}
	
	\ref{eval conv to 0} can be justified, for example, using Theorem \ref{equiv to smooth}. If $\lambda_1(\mathbb{R}^l \times W_l)< 0$ then $\lambda_1(\mathbb{R}^l \times W_l)< -\varepsilon$ for some $\varepsilon>0$. (This can be seen from the proof of Theorem \ref{equiv to smooth}.) Since $\mathbb{R}^l \times W_l$ is the varifold limit of $J_{\#}\circ \tau_{y_m \#}\circ \mu_{r_m \#}|\Sigma_{q_m}|$, in view of Remark \ref{semicontinuity of eigenvalue}, for all large $m$, 
	$$\lambda_1\left(J_{\#}\circ \tau_{y_m \#}\circ \mu_{r_m \#}|\Sigma_{q_m}|\right) < -\varepsilon$$
	This contradicts \ref{eval conv to 0} as $\lim_{m \rightarrow \infty}r_m = \infty$. Hence, $\lambda_1(\mathbb{R}^l \times W_l)\geq 0$ i.e. $\mathbb{R}^l \times W_l$ is stable.
\end{proof}
\medskip
To prove the graphical convergence part of Theorem \ref{main theorem to prove} we will need the following Lemma which is a consequence of Theorem \ref{modification of SS thm-1}. 
\\
\begin{lemma}
	Let $(N^{n+1},g)$ be a closed Riemannian manifold. Let $\{\Sigma_{q}\}$ be a sequence of singular, minimal hypersurfaces in $N$ with $\mathcal{H}^{n-2}(sing(\Sigma_{q}))=0$ for all $q$. We also assume that $W_q = |\Sigma_{q}|$ varifold converges to a stationary, integral varifold $W$, $\Sigma = spt(W)$ and $\Sigma_{q}$ converges to $\Sigma$ in the Hausdorff topology. Then,  for every $x_0 \in reg(\Sigma)$ \te $R=R(x_0)>0$ \st if $r\leq R$ and $\lambda_1(W_q \mres B^N(x_0,r)) \geq -\alpha$ for all $q$ sufficiently large then $\Sigma_q \cap B^N(x_0,r/4)$ is smooth for large $q$ and $\Sigma_q$ converges to $\Sigma$ smoothly in the ball $B^N(x_0,r/5)$.
	\label{modified SS smooth convergence}
\end{lemma}
\begin{proof}
	Let $ \Theta =\lim_{t \rightarrow 0} \frac{\|W\|(B^N(x_0,t))}{t^n}$ . (The limit exists since $W$ is stationary.) Let $\rho_0$ be \st
	\begin{equation}
	\frac{\|W\|(B^N(x_0,t))}{t^n}\leq \Theta +1/2 \qquad \forall t\leq \rho_0
	\label{density bound}
	\end{equation}
	Let $\delta_0$ and $r_0$ be the constants which are provided by Theorem \ref{modification of SS thm-1} when we set $\mu = \Theta +1$ and $\mu_1 = 1$. Let $s_0 = \frac{\delta_0 r_0}{2}$. For $0<a<\frac{1}{2}\text{inj}_N$, identifying $\bbr^{n+1}$ with $T_{x_0}N$ (and $\bbr^n \times \{0\} \subset \bbr^{n+1}$ is identified with $ T_{x_0}M$) we define
	$$\Phi_a : B^{n+1}(0,2) \longrightarrow B^N(x_0,2a), \quad \Phi_a(v)= \text{ exp}_{x_0}(av), \quad g_a=\Phi_a^*g $$
	
	We can choose $R>0$ so that 
	\medskip
	\begin{itemize}
		\item $R < \min\left\lbrace\frac{1}{2}\text{inj}_N, \frac{1}{2}d^N(x_0,sing(\Sigma)), 1, \rho_0\right\rbrace$
	\end{itemize} 
	\medskip 
	and whenever $r\leq R$
	\medskip
	\begin{itemize}	
		\item $(B^{n+1}(0,1),g_r)$ satisfies \ref{C^2 bound on metric} with $\mu_1 = 1$ \label{mu_1 = 1}
		\item $\tilde{\Sigma}=\Phi_r^{-1}(\Sigma \cap B^N(x_0,2r)) \subset B^{n+1}(0,2) \cap \{x:|x_{n+1}| < s_0\}$ \footnote{This is possible since $\Sigma$ is smooth near $x_0$}
	\end{itemize}
	\medskip
	We will show that the above choice of $R= R(x_0)$ works. Let us fix an $r \leq R$. We define 
	$$C=B^n(0,1) \times [-1,1] \subset B^{n+1}(0,2), \qquad C'= C \setminus \left\{x:|x_{n+1}| < s_0\right\}.$$
	Then 
	$$C' \cap \left(\tilde{\Sigma} \cup \partial  B^{n+1}(0,2)\right) = \emptyset; \qquad d:= \text{dist}_{g_r}\left(C', \left(\tilde{\Sigma} \cup \partial  B^{n+1}(0,2)\right)\right)$$
	We choose $q_0$ (depending on $r$) so that for all $q \geq q_0$
	\medskip
	\begin{itemize}
		\item $0<\frac{\|W_q\|(B^N(x_0,r))}{r^n}\leq \Theta +1$ \footnote{This is possible because of \ref{density bound} and varifold convergence of $W_q$ to $W$}
		\item $\lambda_1(W_q \mres B^N(x_0,r)) \geq -\alpha$
		\item The Hausdorff distance $d_H(\Sigma_q, \Sigma) < dr$
	\end{itemize}
	\medskip
	Hence, denoting $\tilde{\Sigma}_q=\Phi_r^{-1}(\Sigma_q \cap B^N(x_0,2r))$, for all $q \geq q_0$ we have $\tilde{\Sigma}_q\cap B^{n+1}(0,1) \neq \emptyset$, 
	
	$$ \mathcal{H}^n(\tilde{\Sigma}_q\cap B^{n+1}(0,1))= \frac{\mathcal{H}^n\left(\Sigma_q \cap B^N(x_0,r)\right)}{r^n}\leq \Theta +1, $$
	
	$$ \lambda_1(|\tilde{\Sigma}_q\cap B^{n+1}(0,1)|) \geq -\alpha r\geq -\alpha$$
	
	and
	\begin{equation}
	\tilde{\Sigma}_q \cap C' = \emptyset \quad \text{i.e.} \quad (\tilde{\Sigma}_q \cap C ) \subset C \cap \left\{x:|x_{n+1}| < s_0\right\}\footnote{This follows from $d_H(\Sigma_q, \Sigma) < dr$}
	\label{tilde{Sigma}_q is trapped}
	\end{equation}
	
	We can now apply Theorem \ref{modification of SS thm-1} to the singular, minimal hypersurface $$\left(\tilde{\Sigma}_q \cap B^{n+1}(0,1)\right) \subset \left(B^{n+1}(0,1), g_r\right)$$ 
	for $q \geq q_0$. Since, $B^{n+1}(0,1) \subset C$,
	equation \ref{tilde{Sigma}_q is trapped} implies that for all $q \geq q_0$, $x \in \tilde{\Sigma}_q \cap B^{n+1}(0,1/4)$ and for $\rho = r_0$ the oscillation bound (*) of Theorem \ref{modification of SS thm-1} is satisfied in the cylinder $C(x, r_0)$. Hence, $\tilde{\Sigma}_q \cap B^{n+1}(0,1/4)$ is smooth with uniform bound on $A_{\tilde{\Sigma}_q \cap B^{n+1}(0,1/4)}$ and we have already noted above a uniform upper bound on $\mathcal{H}^n(\tilde{\Sigma}_q \cap B^{n+1}(0,1)) $. The counterpart of the Arzela-Ascoli theorem for smooth, minimal hypersurface gives that in a smaller ball $B^{n+1}(0,1/5)$, $\tilde{\Sigma}_q$ converges to $\tilde{\Sigma}$ smoothly and graphically. (Here we do not have to pass to a further subsequence since we already know that  $\tilde{\Sigma}_q$ Hausdorff converges to $\tilde{\Sigma}$). When we rescale back to go back to $N$, we get that in the ball $B^N(x_0, r/5)$, $\Sigma_q$ converges to $\Sigma$ smoothly and graphically.
	
\end{proof}


\section{Proof of the main theorem}

In this section we will give a proof of Theorem \ref{main theorem to prove}.
\\\\
By Allard’s compactness
theorem \cite{Allard}, possibly after passing to a subsequence, $V_k \longrightarrow V$ in the $\textbf{F}$ metric with $V$ is a stationary, integral varifold and $\|V\|(N)\leq \Lambda$. If $\lambda_p(V)<-\alpha$ then $\lambda_p(V_k)<-\alpha$ for $k$ sufficiently large by Remark \ref{semicontinuity of eigenvalue}. Hence, $\lambda_p(V)\geq -\alpha$. We also know that $\text{spt}(V_k)=M_k$ converges to $\text{spt}(V)=M$ in the Hausdorff topology.
\\
\begin{lemma}
	Let $W$ be a stationary varifold in $U^{n+1}$ with $\lambda_p(W)\geq -\alpha$, $\alpha \geq 0$. Let $U_1, ..., U_p$ be mutually disjoint open subsets of $U$ \st $\|W\|(U_j) \neq 0$ for each $j=1,...,p$. Then \te $i \in \{1,...,p\}$ \st $\lambda_1(W \mres U_i ) \geq -\alpha$.
	\label{lemma disjoint open sets}
\end{lemma}

\begin{proof}
	Let us assume that for all $i = 1,...,p$ we have $\lambda_1(W \mres U_i ) < -\alpha$. Hence, from Definition \ref{Def of jacobi eigenvalue} and equation \ref{Hess A and 2nd var} we have maps
	$$F^i : \overline{B}^1 \longrightarrow \text{Diff}(U_i)$$
	and vector-fields $Y^i$ compactly supported in $U_i$
	$$Y^i\big |_x = \frac{d}{dt}\bigg |_0 F^i_t(x)$$
	\st
	$$\delta^2W(Y^i, Y^i) < -\alpha \|Y^i\|_{L^2}^2$$
	We define 
	$$F : \overline{B}^p \longrightarrow \text{Diff}(U), \qquad F_v = \Phi^{\sum_iv_iY^i}$$
	where $\Phi^X$ denotes the time $1$ flow of the vector field $X$. For this choice of $F$ we have
	\begin{align*}
	D^2A^W\big |_0(u,u) &=\delta^2W\left( \sum_i u_iY^i, \sum_i u_iY^i\right)\\
	& = \sum_i u_i^2\delta^2W(Y^i, Y^i) \quad (\text{since spt}(Y^i) \text{'s are mutually disjoint})\\
	&< -\alpha \sum_i u_i^2\|Y^i\|_{L^2}^2 \\
	& = -\alpha \big \|\sum_i u_iY^i \big \|_{L^2}^2 \quad (\text{since spt}(Y^i) \text{'s are mutually disjoint})\\
	& = -\alpha K^W(u,u)	
	\end{align*}
This contradicts $\lambda_p(W)\geq -\alpha$.	
\end{proof}

Returning back to the proof of the main Theorem \ref{main theorem to prove}, let $G\subset M$ be the set of points $x \in M$ for
which there exists $r= r(x)>0$ and some subsequence $\{V_{k'}\} \subset \{V_k\}$ \st 
for each $k'$, $\lambda_1(V_{k'}\mres B^N(x, 2r))\geq -\alpha$. \tf using Theorem \ref{modification of SS thm-2}, $
\mathcal{H}^{s}\left(sing(M)\cap B^N(x, r(x))\right)=0$ for all $x \in G$ and $ 
s > n-7$.
\\
\begin{lemma}
	
	The set $M \setminus G$ has atmost $p-1$ points.
	\label{bad pt cardinality}
	
\end{lemma}

\begin{proof}
	
	Suppose there exists $p$ points $\{x_i\}_{i=1}^{p}\subset M \setminus G $. Let $t$ be small enough so that the normal geodesic balls $\{B^N(x_i, t)\}_{i=1}^{p}$ are mutually disjoint. By the definition of $G$, there exists $k_0$ \st for all $k\geq k_0$, $\lambda_1(V_k \mres B^N(x_i, t))<-\alpha$ for each $i$. By Lemma \ref{lemma disjoint open sets}, this implies $\lambda_p(V_k)<-\alpha$ for all $k \geq k_0$, a contradiction.
\end{proof}

We note that
\begin{equation}
sing(M)\cap G = \bigcup_{x\in sing(M)\cap G} \left(sing(M)\cap B^N(x, r(x))\right)
\label{union of good balls }
\end{equation}

We can extract a countable subcover of the R.H.S. of \ref{union of good balls } and write\footnote{This is possible because $N$, being a manifold, is second-countable and a subspace of a second-countable space is second-countable.} 

\begin{equation}
sing(M)\cap G = \bigcup_{i=1}^{\infty} \left(sing(M)\cap B^N(x_i, r(x_i))\right)
\label{countable union of good balls}
\end{equation}

where each $x_i$ is in sing$(M) \cap G$. By the definition of $G$, for $s >n-7$, $\mathcal{H}^{s}(sing(M)\cap B^N(x_i, r(x_i)))=0$ for each $i$. \tf \ref{countable union of good balls} and Lemma \ref{bad pt cardinality} imply that $\mathcal{H}^{s}(sing(M))=0$. (Here we are implicitly assuming that $n\geq7$ so that for $s>n-7 \geq 0$, $\mathcal{H}^s(M\setminus G)=0$; when $n<7$ further arguments are required to show that $M$ is smooth at the points of $M \setminus G$ as explained in \cite{ACS}.)
\\\\

We can now complete the proof of graphical convergence. We will produce a set $\mathcal{Y} \subset reg(M)$ and a subsequence $\{M_{k'}\} \subset \{M_k\}$ \st $M_{k'}$ converges smoothly and graphically on compact subsets of $reg(M)\setminus \mathcal{Y}$. Let $X$ be a countable, dense subset of $reg(M)$ and
$$\mathcal{B}=\left \lbrace B^N(x,r): x\in X, r \in \mathbb{Q}^{+}, r<d^N(x, sing(M)), r< \text{inj}_N\right \rbrace$$

Then $\mathcal{B}$ is a countable collection of balls, say, $\mathcal{B} = \{B_i\}_{i=1}^{\infty}$. We will mark each $B_i$ as good or bad and to each $B_i$ we will assign an infinite index set $I_i \subset \bbn $ as follows. At the first step we examine whether there exists an infinite set $J \subset \bbn$ \st $\{M_j\}_{j \in J} $ converges to $M$ smoothly in $B_1$. If it exists we mark $B_1$ as good and define $I_1$ to be that $J$. Otherwise we mark $B_1$ as bad and define $I_1$ to be $\bbn$. Suppose we have marked $B_{i-1}$ as good or bad and defined $I_{i-1}$. Then we examine whether there exists an infinite set $J \subset I_{i-1}$ \st $\{M_j\}_{j \in J} $ converges to $M$ smoothly in $B_i$. If it exists we mark $B_i$ as good and define $I_i$ to be that $J$. Otherwise we mark $B_i$ as bad and define $I_i$ to be $I_{i-1}$. \\\\
Let $\mathcal{G}$ be the union of good balls. Denoting $\mathcal{Y} = reg(M) \setminus \mathcal{G}$, we claim that $|\mathcal{Y}|\leq p-1$. Otherwise, \te $p$ distinct points $x_1,...,x_p$ in $\mathcal{Y}$. To arrive at a contradiction we will apply Lemma \ref{modified SS smooth convergence} to the sequence of singular, minimal hypersurfaces $\{M_k\} $; $|M_k|$ converges in the varifold sense to $V$ which is supported on $M$; Lemma \ref{modified SS smooth convergence} provides a function $R : reg(M) \longrightarrow \bbr^+$.  Let $\tau$ be a positive number \st $5\tau \leq R(x_l)$ for each $l$ and the balls
 $\{B^N(x_l, 5\tau)\}_{l=1}^p$ are mutually disjoint. There exists
  $B_{i_l} \in \mathcal{B}$ \st $x_l \in B_{i_l} \subset B^N(x_l, \tau)$.
  As $x_l \in \mathcal{Y}$, $B_{i_l}$ is a bad ball. Hence $\{M_j\}_{j \in I_{i_l}}$ does not have a subsequence which smoothly converges to $M$
  in $B^N(x_l, \tau)$. \tf by Lemma \ref{modified SS smooth convergence}, $\lambda_1(V_j \mres B^N(x_l , 5\tau)) < -\alpha$ for all large $j \in I_{i_l}$. Without loss of generality, we can assume that
 $i_1 <...<i_p$ so that $I_{i_1} \supset ... \supset I_{i_p}$. Hence $\lambda_1(V_j \mres B^N(x_l , 5\tau)) < -\alpha$ for each $l=1,..,p$ and for all large $j \in I_{i_p}$. By Lemma \ref{lemma disjoint open sets} this gives $\lambda_p(V_{j}) < -\alpha$ for all large $j \in I_{i_p}$, a contradiction. Hence, $|\mathcal{Y}| \leq p-1$.
\\\\
By a diagonal argument, we can choose an infinite set $I \subset \mathbb{N}$ \st $|I \setminus I_i|$ is finite for all $i$. Then, by the definition of good ball, $\{M_i\}_{i \in I}$ is a sequence which converges smoothly and graphically on the compact subsets of $reg(M) \cap \mathcal{G} =reg(M)\setminus \mathcal{Y}$. 

\section{Some further remarks}
Besides the main compactness Theorems of \cite{Sharp} and \cite{ACS}, some additional results proved in these two papers can be suitably generalized in higher dimensions. In this last section, we will state them as a sequence of remarks. Below we will assume that $M_k$'s and $M$ are as in Theorem \ref{main theorem to prove} (and Theorem \ref{sharp}).
\medskip
\begin{remark}
	We have assumed that each $M_k$ is connected. Since $\{M_k\}$ converges to $M$ in the Hausdorff distance, this implies that $M$ is connected as well. From \cite{Ilmanen} (Theorem A (ii)) it follows that $reg(M)$ and hence $reg(M) \setminus \mathcal{Y}$ is also connected. Therefore, the number of sheets in the graphical convergence is constant over $reg(M) \setminus \mathcal{Y}$. In particular, $V= m |M|$ for some $m \in \bbn.$
\end{remark}
\bigskip
\begin{remark}
	If the number of sheets in the graphical convergence is $1$, then $\mathcal{Y} = \emptyset$. This is Claim $4$ in \cite{Sharp} and the proof presented there works in our case as well.
\end{remark}
\bigskip
\begin{remark}
	Suppose $reg(M)$ is two sided. If the number of sheets in the graphical convergence is at least $2$ or if the number of sheets is $1$ and $M_k \cap M = \emptyset$ for large $k$, we can construct a positive Jacobi field on $reg(M)$. In this case, $reg(M)$ and hence $M$ is stable. The proof is same as presented in \cite{Sharp}.
	\label{jacobi field}
\end{remark}
\bigskip
\begin{remark}
	Continuing with Remark \ref{jacobi field}, suppose $Ric(N,g)>0$. Then the convergence of $M_k$ to $M$ is always single sheeted. This can be thought of as a higher dimensional analogue of Choi-Schoen \cite{CS} which asserts that in a three \mf with possitive Ricci curvature, the space of closed, embedded minimal surfaces with bounded genus is compact in the smooth topology.  Indeed, in our case, if $H_n(N, \mathbb{Z}_2)=0$ then $reg(M)$ is  two sided; hence, if the number of sheets is $\geq 2$, $M$ is stable by the above Remark \ref{jacobi field}. However, as proved in \cite{Zhou} (Lemma 2.8) positive Ricci curvature of $(N,g)$ implies that $M$ can not be stable. The general case can be obtained by lifting $M_k$'s and $M$ to the universal cover $\tilde{N}$ of $N$ (by \cite{F} and Lemma 2.10 of \cite{Zhou} the lifts $\tilde{M}_k , \tilde{M}$ are connected).
\end{remark}
\bigskip
\begin{remark}
	Theorems \ref{main theorem to prove} and \ref{sharp} hold in the varying \mt set-up. More precisely, instead of assuming $|M_k|$ is stationary with respect to the fixed \mt $g$ if we assume that $|M_k|$ is stationary with respect to the \mt $g_k$ and $g_k$ converges to $g$ in $C^3$, Theorems \ref{main theorem to prove} and \ref{sharp} continue to hold.
\end{remark}
\nocite{*}
\bigskip

\bibliographystyle{amsalpha}
\bibliography{higherdimcptness}

\begin{bibdiv}
\begin{biblist}

\bib{ACS}{article}{
      author={Ambrozio, L.},
      author={Carlotto, A.},
      author={Sharp, B.},
       title={Compactness of the space of minimal hypersurfaces with bounded
  volume and p-th jacobi eigenvalue},
        date={2016},
     journal={J. Geom. Anal.},
      volume={4},
       pages={2591\ndash 2601},
}

\bib{Allard}{article}{
      author={Allard, W.~K.},
       title={On the first variation of a varifold},
        date={1972},
     journal={Ann. of Math.},
      volume={95(3)},
       pages={417–491},
}

\bib{Alm}{book}{
      author={Almgren, F.},
       title={The theory of varifolds},
   publisher={Mimeographed notes, Princeton},
        date={1965},
}

\bib{CM}{article}{
      author={Chodosh, O.},
      author={Mantoulidis, C.},
       title={Minimal surfaces and the allen-cahn equation on 3-manifolds:
  index, multiplicity, and curvature estimates},
        date={2018},
     journal={arXiv:1803.02716 [math.DG]},
}

\bib{CS}{article}{
      author={Choi, H.~I.},
      author={Schoen, R.},
       title={The space of minimal embeddings of a surface into a
  threedimensional manifold of positive ricci curvature},
        date={1985},
     journal={Invent. Math.},
      volume={81},
       pages={387–394},
}

\bib{F}{article}{
      author={Frankel, T.},
       title={On the fundamental group of a compact minimal submanifold},
        date={1966},
     journal={Ann. of Math.},
      volume={83 (1)},
       pages={68–73},
}

\bib{Gaspar}{article}{
      author={Gaspar, Pedro},
       title={The second inner variation of energy and the morse index of limit
  interfaces},
        date={2017},
     journal={arXiv:1710.04719 [math.DG]},
}

\bib{GG}{article}{
      author={Gaspar, P.},
      author={Guaraco, M. A.~M.},
       title={The weyl law for the phase transition spectrum and the density of
  minimal hypersurfaces},
        date={2018},
     journal={arXiv:1804.04243 [math.DG]},
}

\bib{H}{article}{
      author={Hiesmayr, F.},
       title={Spectrum and index of two-sided allen-cahn minimal
  hypersurfaces},
        date={2017},
     journal={arXiv:1704.07738 [math.DG]},
}

\bib{Ilmanen}{article}{
      author={Ilmanen, T.},
       title={A strong maximum principle for singular minimal hypersurfaces},
        date={1996},
     journal={Calc. Var.},
      volume={4},
       pages={443\ndash 467},
}

\bib{IMN}{article}{
      author={Irie, K.},
      author={Marques, F.~C.},
      author={Neves, A.},
       title={Density of minimal hypersurfaces for generic metrics},
        date={2018},
     journal={Ann. of Math.},
      volume={187},
       pages={963\ndash 972},
}

\bib{Li}{article}{
      author={Li, Y.},
       title={Existence of infinitely many minimal hypersurfaces in
  higher-dimensional closed manifolds with generic metrics},
        date={2019},
     journal={arXiv:1901.08440 [math.DG]},
}

\bib{LMN}{article}{
      author={Liokumovich, Y.},
      author={Marques, F.~C.},
      author={Neves, A.},
       title={Weyl law for the volume spectrum},
        date={2018},
     journal={Ann. of Math.},
      volume={187},
       pages={933\ndash 961},
}

\bib{MN_ricci_positive}{article}{
      author={Marques, F.~C.},
      author={Neves, A.},
       title={Existence of infinitely many minimal hypersurfaces in positive
  ricci curvature},
        date={2017},
     journal={Invent. Math.},
      volume={209},
      number={2},
       pages={577–616},
}

\bib{MN_index}{article}{
      author={Marques, F.~C.},
      author={Neves, A.},
       title={Morse index and multiplicity of min-max minimal hypersurfaces},
        date={2016},
     journal={Cambridge J. Math.},
      volume={4},
      number={4},
       pages={463\ndash 511},
}

\bib{MNS}{article}{
      author={Marques, F.~C.},
      author={Neves, A.},
      author={Song, A.},
       title={Equidistribution of minimal hypersurfaces in generic metrics},
        date={2017},
     journal={arXiv:1712.06238 [math.DG]},
}

\bib{Pitts}{book}{
      author={Pitts, J.},
       title={Existence and regularity of minimal surfaces on riemannian
  manifolds},
   publisher={Mathematical Notes 27, Princeton University Press, Princeton,},
        date={1981},
}

\bib{Sharp}{article}{
      author={Sharp, B.},
       title={Compactness of minimal hypersurfaces with bounded index},
        date={2017},
     journal={J. Differential Geom.},
      volume={106},
       pages={317–339},
}

\bib{Sim}{unpublished}{
      author={Simon, L.},
       title={Lectures on geometric measure theory},
        note={Proceedings of the Centre for Mathematical Analysis, Australian
  National University, Canberra, (1983)},
}

\bib{song}{article}{
      author={Song, A.},
       title={Existence of infinitely many minimal hypersurfaces in closed
  manifolds},
        date={2018},
     journal={arXiv:1806.08816 [math.DG]},
}

\bib{SS}{article}{
      author={Schoen, R.},
      author={Simon, L.},
       title={Regularity of stable minimal hypersurfaces},
        date={1981},
     journal={Comm. Pure Appl. Math.},
      volume={34},
       pages={741–797},
}

\bib{white2}{article}{
      author={White, B.},
       title={On the bumpy metrics theorem for minimal submanifolds},
        date={2017},
     journal={Amer. J. Math},
      volume={139(4)},
       pages={1149–1155},
}

\bib{white}{article}{
      author={White, B.},
       title={The space of minimal submanifolds for varying riemannian
  metrics},
        date={1991},
     journal={Indiana Univ. Math. J.},
      volume={40},
       pages={161–200},
}

\bib{Zhou}{article}{
      author={Zhou, X.},
       title={Min-max hypersurface in manifold of positive ricci curvature},
        date={2017},
     journal={J. Differential Geom.},
      volume={105},
      number={2},
       pages={291–343},
}

\bib{Zhou2}{article}{
      author={Zhou, X.},
       title={On the multiplicity one conjecture in min-max theory},
        date={2019},
     journal={arXiv:1901.01173 [math.DG]},
}

\end{biblist}
\end{bibdiv}

\end{document}